\documentclass[reqno,11pt]{amsart}
\usepackage{latexsym,amssymb,amsfonts,amsmath,amsthm}
\usepackage[usenames]{xcolor}
\usepackage{eucal}
\usepackage{mathtools}
\usepackage[colorlinks,citecolor=blue,urlcolor=blue]{hyperref}
\newtheorem{thm}{Theorem}[section] 
\newtheorem{cor}[thm]{Corollary}
\newtheorem{lem}[thm]{Lemma}
\newtheorem{prop}[thm]{Proposition}
\theoremstyle{definition}

\newtheorem{example}[thm]{Example}
\theoremstyle{remark}
\newtheorem{rem}[thm]{Remark}
\numberwithin{equation}{section}
\newcommand{\Real}{\mathbb R}
\newcommand{\eps}{\varepsilon}

\newcommand{\F}{\mathcal{F}}

\renewcommand{\P}{\mathbb{P}}

\newcommand{\E}{\mathbb{E}}

\DeclareMathOperator*{\sign}{sign}

\renewcommand{\Re}{\mathrm{Re}}
\renewcommand{\Im}{\mathrm{Im}}
\def\XXint#1#2#3{{\setbox0=\hbox{$#1{#2#3}{\int}$}
     \vcenter{\hbox{$#2#3$}}\kern-.5\wd0}}

\makeatletter
\newcommand{\doublewidetilde}[1]{{%
  \mathpalette\double@widetilde{#1}%
}}
\newcommand{\double@widetilde}[2]{%
  \sbox\z@{$\m@th#1\widetilde{#2}$}%
  \ht\z@=.9\ht\z@
  \widetilde{\box\z@}%
}
\makeatother

\makeatletter
\let\@mkboth\@gobbletwo
\let\@oddhead\@empty
\let\@evenhead\@empty
\makeatother

\begin{document}

\title[The mixed fBm: a spectral take]
{Mixed fractional Brownian motion: a spectral take}

\author{P. Chigansky}%
\address{Department of Statistics,
The Hebrew University,
Mount Scopus, Jerusalem 91905,
Israel}
\email{pchiga@mscc.huji.ac.il}

\author{M. Kleptsyna}%
\address{Laboratoire de Statistique et Processus,
Universite du Maine,
France}
\email{marina.kleptsyna@univ-lemans.fr}

\author{D. Marushkevych}%
\address{Laboratoire de Statistique et Processus,
Universite du Maine,
France}
\email{dmytro.marushkevych.etu@univ-lemans.fr}

\thanks{P. Chigansky's research was funded by ISF  1383/18
grant}
%\subjclass{}%

\date{\today}%
%\dedicatory{}%
%\commby{}%
% ----------------------------------------------------------------
\begin{abstract}
This paper provides yet another look at the mixed fractional Brownian motion (fBm), 
this time, from the spectral perspective. We derive an approximation for the eigenvalues of its covariance operator,  
asymptotically accurate up to the second order. This in turn allows to compute the exact $L_2$-small ball probabilities, 
previously known only at logarithmic precision. The obtained expressions show an interesting stratification of scales, 
which occurs at certain values of the Hurst parameter of the fractional component.  
Some of them have been previously encountered in other problems involving such mixtures. 
\end{abstract}

\maketitle 

\section{Introduction}

Mixtures of stochastic processes can often have properties, different from the individual components.  
In this paper we revisit the {\em mixed}  fractional Brownian motion (fBm)
\begin{equation}\label{mfBm}
\widetilde B_t = B_t + B^H_t, \quad t\in[0,1]
\end{equation}
where $B_t$ and $B^H_t$ are independent standard and fractional Brownian motions, respectively. The latter is 
the centred Gaussian process with the covariance function 
$$
\E B^H_t B^H_s = \frac 1 2 \Big(t^{2H}+s^{2H}-|t-s|^{2H}\Big), 
$$
where $H\in (0,1)$ is a parameter, called the Hurst index. Introduced in \cite{Kol40}, \cite{MvN68},  nowadays the fBm takes central place in the study of heavy tailed distributions, 
self-similarity and long range dependence phenomenon \cite{EM02}, \cite{PT17}.

The mixture \eqref{mfBm} drew considerable attention, after some of its interesting properties have been discovered in \cite{Ch01} (see also 
\cite{BN03}, \cite{CCK}) and proved useful in mathematical finance \cite{Ch03}, \cite{BSV07} and statistical inference 
\cite{CKstat}, \cite{DMS14}. It turns out that
the process $\widetilde B$ is a semimartingale if and only if $H> \frac 34$, in which case it is measure equivalent to $B$;
for $H<\frac 1 4$ the mixed fBm is measure equivalent to $B^H$, see \cite{vZ07}.

Another interesting feature is revealed by its canonical representation from \cite{CCK},   based on the  
martingale  
$$
M_t = \E(B_t|\F^{\widetilde B}_t) = \int_0^t g(s,t)d\widetilde B_s, \quad t>0, 
$$
where the kernel $g(s,t)$ is obtained by solving certain Wiener-Hopf equation. 
%\begin{equation}\label{inteq}
%g(s,t) +  \int_0^t g(r,t)  c_H |r-s|^{2H-2} dr = 1, \quad 0\le s\le t\le 1,
%\end{equation}
%with the constant $c_H$. 
The limiting performance of statistical procedures in models involving mixed fBm is governed by the 
asymptotic behaviour of this equation as $t\to\infty$, see \cite{CKstat}. 
For $H>\frac 12$, under reparametrization $\eps:= t^{1-2H}$ it reduces to the singularly perturbed problem 
\begin{equation}\label{singpert}
\eps g_\eps(x) +\int_0^1 g_\eps(y) c_H |x-y|^{2H-2}dy = 1, \quad 0\le x\le 1.
\end{equation}

As $\eps\to 0$ its solution $g_\eps(x)$ converges to the solution $g_0(x)$ of the limit equation, obtained by 
setting $\eps:=0$ in \eqref{singpert}, with the following rate with respect to $L_2$-norm, see \cite{ChK}, 
$$
\|g_\eps-g_0\|_2 \sim   \begin{cases}
\eps^{\frac{1-H}{2H-1}} & H\in (\frac 2 3,1)\\
\eps \sqrt{\log \eps^{-1}} & H=\frac 2 3\\
\eps & H\in (\frac 1 2, \frac 2 3)
\end{cases}
$$
Here the convergence rate breaks down at the critical point $H=\frac 2 3$.
%Curiously, it also arises in the filtering problem of estimating $B^H_t$ given the ``white noise'' observations
%$$
%Y_t = \int_0^t B^H_s ds + B_t, \quad t\ge 0.
%$$
%The steady state filtering error, found in \cite{ChK}, 
%$$
%V_\infty(H):= \lim_{t\to\infty}\var (B^H_t|\F^Y_t) = \frac{\big(\sin (\pi H)\Gamma(2H+1)\big)^{\frac 1{2H+1}}}{\sin \frac \pi {2H+1}},
%$$
%is maximized at $H=\frac 23$.

The purpose of this paper is to demonstrate that the mixed fBm is also an interesting process from the spectral standpoint. 
Using the results from the general operator theory, see  \cite[Proposition 2.3]{NN04tpa}, the first order asymptotics of the 
eigenvalues of its covariance operator coincides with that of its ``slowest" component, while the other component acts as a perturbation. 
More precisely, for $H>\frac1 2$ its eigenvalues agree, to the first asymptotic order, with the standard Brownian motion, and, for $H<\frac 1 2$, with the fBm. It can hardly be expected that such separation is preserved on a finer asymptotic level; however, 
our main result shows that the two components do remain separated, albeit in a somewhat unexpected way,  see Theorem \ref{eig-thm} and Remark \ref{rem:2.2} below.

As an application of the spectral approximation derived in this paper, we consider the $L_2$-small ball probabilities problem 
of evaluating $\P(\|\widetilde B\|_2\le \eps)$ as $\eps\to 0$. 
We show that it exhibits a curious stratification of scales, which occurs at the following values of $H$ (see Theorem \ref{main-thm}):
$$
{\bf \frac 1 4}, \frac 1 3, \frac 3 8, ..., \frac 1 2, ..., \frac 5 8, {\bf \frac 2 3}, {\bf \frac 3 4}
$$
where the values, mentioned above, are emphasized in bold.

\section{Main results} 
Spectral theory of stochastic processes is a classical theme in probability and analysis. 
For a centred process $X=(X_t, t\in [0,T])$ with covariance $\E X_s X_t = K(s,t)$, the eigenproblem is 
to find the nontrivial solutions $(\lambda, \varphi)$ to the equation 
\begin{equation}\label{eigpr}
K \varphi - \lambda \varphi =0,
\end{equation}
where $K$ is the integral covariance operator 
$$
(K \varphi)(t) = \int_0^T K(s,t) \varphi(s)ds.
$$

For sufficiently regular kernels this problem has countable solutions $(\lambda_n, \varphi_n)_{n\in \mathbb{N}}$.
The ordered sequence of the eigenvalues $\lambda_n\in \Real_+$ converges to zero and the eigenfunctions $\varphi_n$ form 
an orthonormal basis in $L_2([0,T])$. The Karhunen-Lo\`{e}ve theorem asserts that $X$ can be expanded into series of the 
eigenfunctions 
$$
X_t = \sum_{n=1}^\infty Z_n \sqrt{\lambda_n} \varphi_n(t)
$$
where $Z_n$'s are uncorrelated zero mean random variables with unit variance.
This decomposition is  useful in many applications, if the eigenvalues and the eigenfunctions can be found in a closed form or,
at least, approximated to a sufficient degree of accuracy, see e.g. \cite{PT17}. 
There are only a few processes, however, for which the eigenproblem can be solved explicitly. 

\subsection{Eigenvalues of the mixed fBm}
 
One such process is the Brownian motion, for which a simple exact formula is long known:
\begin{equation}\label{eigBm}
\lambda_n = \frac 1 {(n-\frac 12 )^2\pi^2}\quad \text{and}\quad \varphi_n(t) = \sqrt{2} \sin \big((n-\tfrac 12 )\pi t\big).
\end{equation}
The eigenproblem for the fBm turns out to be much harder and it is unlikely to 
have any reasonably explicit solutions. Nevertheless, in this case the eigenvalues admit a fairly precise asymptotic 
approximation. Namely, the sequence of ``frequencies'' $\nu_n(H)$, defined by the relation
\begin{equation}\label{lambdaH}
\lambda_n(H) = \frac{\sin (\pi H)\Gamma(2H+1)}{\nu_n(H)^{2H+1}}
\end{equation}
has the asymptotics
\begin{equation}\label{nun}
\nu_n(H) = \Big(n-\frac 1 2\Big)\pi -   \frac{(H-\frac 12)^2 }{H+\frac 12}\frac \pi 2+ O(n^{-1}), \qquad n\to\infty. 
\end{equation}
The leading order term in \eqref{nun} was discovered in \cite{Br03a, Br03b} and, by different methods, in \cite{NN04tpa} and \cite{LP04};
the second term was recently obtained in \cite{ChK}, along with the following approximation for the eigenfunctions 
\begin{multline}\label{fBmeigfn}
\varphi_n(x)  
= 
 \sqrt 2 \sin\bigg( \nu_{n}(H) x+\frac 1 4 \frac{(H-\frac 1 2)(H-\frac 3 2)}{H+\frac 1 2}\bigg) \\
 -    \int_0^{\infty}    
\Big(
e^{-  x \nu_n(H) u} f_0(u) +
(-1)^{n}   e^{-  (1-x) \nu_n(H) u}f_1(u)\
\Big)du + O(n^{-1}),
\end{multline} 
where $f_0(u)$ and $f_1(u)$ are given by closed form formulas and the residual term is uniform over $x\in [0,1]$.

\medskip
The following result details the spectral asymptotics of the mixed fBm \eqref{mfBm}:

\begin{thm}\label{eig-thm}
Let $\widetilde \lambda_n$ be the ordered sequence of eigenvalues of the mixed fBm covariance operator. 
Then the unique roots $\widetilde\nu_n$ of the equations
\begin{equation}\label{mixed_lambdan}
\widetilde\lambda_n   = \frac {1}{\widetilde \nu_n^2} + \frac{\sin (\pi H)\Gamma(2H+1)}{\widetilde\nu_n^{2H+1}}, \qquad n=1,2,...
\end{equation}
satisfy 
\begin{equation}\label{mnun}
\widetilde\nu_n := \nu_n(\tfrac 1 2\wedge H)  + O\big(n^{-|2H-1|}\big), \qquad n\to \infty,
\end{equation}
where $\nu_n(\cdot)$ is defined in \eqref{nun}.
\end{thm}

\begin{rem} \label{rem:2.2}
This theorem reveals a curios feature in the spectral structure of  mixtures.
In the pure fractional case, the second order approximation for the frequencies $\nu_n$ in \eqref{nun}
furnishes an approximation for the eigenvalues $\lambda_n$ through \eqref{lambdaH}, precise up to the same, 
second order. In the mixed case, similar approximation for the frequencies $\widetilde \nu_n$ in \eqref{mnun} 
provides an approximation for the eigenvalues \eqref{mixed_lambdan}, accurate up to the {\em fourth} order:
for example, for $H>\frac 1 2$
$$
\widetilde \lambda_n = a_1(H) n^{-2} + a_2(H) n^{-2H-1} + a_3(H) n^{-3} + a_4(H) n^{-2H-2} + o(n^{-2H-2})
$$
where all the coefficients $a_j(H)$ can be computed exactly. 
\end{rem}

\begin{rem}
It can be shown that asymptotic behaviour of the eigenfunctions is also dominated by one of the components: 
for $H>\frac 1 2$, the first order asymptotics with respect to the uniform norm 
coincides with that of the standard Brownian motion \eqref{eigBm}, while for $H<\frac 1 2$, it agrees 
with the asymptotics \eqref{fBmeigfn} of the fBm. 
\end{rem}

\subsection{The small ball probabilities} 
The small ball probabilities problem is to 
find the asymptotics of 
\begin{equation}\label{P}
\P(\|X\|\le \eps),\quad   \eps\to 0,
\end{equation}
for a given process $X=(X_t, t\in [0,1])$ and a norm $\|\cdot\|$. 
It has been extensively studied in the past and was found to have deep connections to various topics in probability theory
and analysis, see \cite{LiS01}. The case of the Gaussian processes and the $L_2$-norm is the simplest, in which asymptotics of \eqref{P} is  determined by the eigenvalues $\lambda_n$ of the covariance operator, \cite{S74}.

The computations for concrete processes require a closed form formula or at 
least a sufficiently accurate approximation of the eigenvalues.
Typically,  the first order approximation of $\lambda_n$'s allows to compute the asymptotics of $\log \P(\|X\|_2\le \eps)$ and the second order 
suffices for finding the asymptotics of $\P(\|X\|_2\le \eps)$, exact up to a multiplicative ``distortion'' constant. 
For the fBm, formulas \eqref{lambdaH}-\eqref{nun} give 
$$
\P(\|B^H\|_2\le \eps) \sim \eps^{\gamma(H)} \exp \Big(-\beta(H) \eps^{-\frac 1 H}\Big), \quad \eps \to 0,
$$
where $f(\eps)\sim g(\eps)$ means that $\lim_{\eps\to 0}f(\eps)/g(\eps)$ is finite and nonzero.
The exponent $\beta(H)$ was derived first in \cite{Br03a, Br03b}:
\begin{equation}\label{beta0H}
\beta(H) = 
H
\left(
\frac{ \sin (\pi H)\Gamma(2H+1) }{  (2H+1)^{ 2H+1 }\left( \sin \frac \pi {2H+1} \right)^{ 2H+1 } }
\right)^{\frac 1 {2H}}
\end{equation} 
and the power $\gamma(H)$ was recently found in \cite{ChK}:
\begin{equation}\label{gammaH}
\gamma(H) = \frac 1{2H} \Big(\frac 5 4-H+H^2\Big).
\end{equation}

%For $H=\frac 1 2$ these formulas agree with the well known asymptotics for the Brownian motion: 
%$$
%\P(\|B\|_2\le \eps) \simeq \frac 4{\sqrt \pi}\eps \exp \Big(-\frac 1 8 \eps^{-2}\Big), \quad \eps \to 0.
%$$

The first order perturbation effect, mentioned in the Introduction, implies that the rough, logarithmic asymptotics of 
$L_2$-small ball probabilities for the mixed fBm coincides with either standard or fractional parts 
\cite{NN04tpa} (see also \cite{NN18},\cite{McKMM18}):
\begin{equation}\label{logP}
\log \P\big(\|\widetilde B\|_2 \le \eps\big) \simeq 
 \log \P\big(\|B^{H\wedge \frac 1 2}\|_2\le \eps\big), \quad \eps\to 0,
\end{equation}
where $f(\eps)\simeq g(\eps)$ means that for $\lim_{\eps\to 0}f(\eps)/g(\eps)= 1$.
The following theorem  shows that the exact asymptotics of the mixed fBm is 
more intricate than could have been expected in view of \eqref{logP}:

\begin{thm}\label{main-thm} For $H\in (0,1)\setminus \{\tfrac 1 2\}$ 
\begin{equation}\label{Peps}
\P\big(\|\widetilde B\|_2   \le \eps\big)   \sim     \eps^{ \gamma(H)\vee 1}\exp\bigg(
-\sum_{k=0}^{\big\lfloor \frac 1 {|2H-1|}\big\rfloor}\beta_{k}(H) \eps^{\frac{k|2H-1| -1}{1/2\wedge H}} 
 \bigg), \quad \eps \to 0,
\end{equation}
where $\gamma(H)$ is given by \eqref{gammaH} and $\beta_k(H)$ are positive constants, defined in 
Propositions \ref{main-thm-largeH} and \ref{main-thm-smallH} below.
\end{thm}

\noindent 
A similar type of asymptotics, where additional terms join the sum in the exponent at certain values of a parameter, 
has been recently observed in \cite{R18} for certain Gaussian random fields.

\section{Proof of Theorem \ref{eig-thm}}

The proof uses the program of the spectral analysis for covariance operators, developed in \cite{ChK}, \cite{Ukai}.
It is based on the reduction of the eigenproblem to finding functions $(\Phi_0, \Phi_1)$,
sectionally holomorphic on $\mathbb{C}\setminus \Real_+$, which satisfy 

\medskip

\begin{enumerate}\addtolength{\itemsep}{0.25\baselineskip}
\renewcommand{\theenumi}{\alph{enumi}}

\item {\em a priori} growth estimates at the origin;

\item\label{cond_c} constraints on their values at certain points on the imaginary axis;

\item boundary conditions on the real semi-axis $\Real_+$;

\end{enumerate}

\medskip
\noindent
and behave as polynomials at infinity. The coefficients of these polynomials are also determined   
by condition \eqref{cond_c}.  The asymptotics in the spectral problem is derived by analysis of these two functions.

Implementation of this program uses the technique of solving the Riemann boundary value problem, see \cite{Gahov}.
We will detail its main steps, referring the reader to the relevant parts in \cite{ChK}, whenever calculations are similar.

\subsection{An equivalent generalized spectral problem}

Our starting point is to rewrite \eqref{eigpr} with the kernel 
\begin{equation}\label{Ktilde}
\widetilde K(s,t) = s\wedge t + \frac {1} 2 \Big(s^{2-\alpha}+t^{2-\alpha}-|t-s|^{2-\alpha}\Big), \quad s,t\in [0,1],
\end{equation}
where we defined  $\alpha:=2-2H\in (0,2)$, as the generalized spectral eigenproblem 
\begin{equation} \label{geneig}
\begin{aligned}
&   (1-\tfrac\alpha 2) \frac d {dx} \int_0^1  |x-y|^{1-\alpha}\sign(x-y) \psi(y)dy = -\lambda \psi''(x)-\psi(x),   \\
& \psi(1)=0, \, \psi'(0)=0,
\end{aligned}
\end{equation}
for $\psi(x) := \int_x^1 \varphi(y)dy$. The advantage of looking at the problem in this form is that it involves 
a simpler {\em difference} kernel. The proof of \eqref{geneig} amounts to taking derivatives of \eqref{eigpr} 
and rearranging, see  \cite[Lemma 5.1]{ChK}.

\subsection{The Laplace transform}

The principal stage of the proof is to derive an expression for the Laplace transform  of a solution $\psi$ to \eqref{geneig}:
$$
\widehat \psi(z) = \int_0^1 e^{-zx} \psi(x)dx, \quad z\in \mathbb{C},
$$
which has removable singularities: 

\begin{lem}\label{lem:3.1}
Define the structural function of the problem \eqref{geneig}
$$
\Lambda(z) =   \frac{ \Gamma(\alpha)}{|c_\alpha|}
\Big(  
\lambda    + \frac 1 {z^2} + \kappa_\alpha z^{\alpha-3}  e^{\pm \frac{1-\alpha} 2\pi i} 
\Big), \quad z\in \mathbb{C}\setminus \Real
$$
where $c_\alpha = (1-\frac  \alpha 2)(1-\alpha)$ and 
$
\kappa_\alpha := \dfrac{c_\alpha}{\Gamma(\alpha)}\dfrac{\pi}{\cos \frac \pi 2 \alpha}
$
and the signs correspond to $\Im\{z\}>0$ and $\Im\{z\}<0$ respectively. Then the Laplace transform 
can be expressed as  
$$
\widehat \psi(z) = \frac{\Phi_0(z)+e^{-z} \Phi_1(-z)}{z^2 \Lambda(z)}
$$
where $(\Phi_0, \Phi_1)$ are functions, sectionally holomorphic on $\mathbb{C}\setminus \Real_+$, such that 
\begin{equation}\label{near_zero_est}
\Phi_j(z) = \begin{cases}
O(z^{\alpha-1}) & \alpha <1 \\
O(1) & \alpha >1
\end{cases} \quad \text{as}\ z\to 0 \quad \text{for\ } j=0,1
\end{equation}
and  
\begin{equation}\label{near_inf_est}
\begin{aligned}
&
\Phi_0 (z) = -2C_2 z+\begin{cases}
O(z^{-1}) & \alpha<1 \\
O(z^{\alpha-2}) & \alpha>1
\end{cases} \\
&
\Phi_1 (z) = -2C_1+\begin{cases}
O(z^{-1}) & \alpha<1 \\
O(z^{\alpha-2}) & \alpha>1
\end{cases}
\end{aligned}\qquad \text{as\ } z\to\infty
\end{equation}
where $C_1$ and $C_2$ are constants.

\end{lem}

The proof of this lemma is close to  \cite[Lemma 5.1]{ChK}. 
\subsection{Reduction to integro-algebraic system}
The function $\Lambda(z)$ has zeros at $\pm i\nu$ where $\nu\in \Real_+$ solves 
the equation 
\begin{equation}\label{lambdanu}
\lambda =  \nu^{-2}      +    \kappa_\alpha  \nu^{ \alpha-3}. 
\end{equation}
This defines the one-to-one correspondence between $\lambda$ and $\nu$, cf. \eqref{mixed_lambdan}, with $\nu$ playing the role of a large 
parameter in what follows. Since $\Lambda(\pm i\nu)=0$ and $\widehat \psi(z)$ must be analytic, we obtain the algebraic condition 
\begin{equation}\label{poles}
e^{-i\nu }\Phi_1(-i\nu )+\Phi_0(i\nu )=0.
\end{equation}

Also $\Lambda(z)$ is discontinuous on $\Real_+$. 
Therefore continuity of $\widehat \psi(z)$ on  $\mathbb{C}$ gives the boundary conditions, which bind together the limits 
$\Phi_j^\pm := \lim_{z\to t^\pm} \Phi_j(z)$ as $z$ tends to $t\in \Real$ in the lower and upper half planes. 
In the vector form (see \cite[Section 5.1.2 ]{ChK}) 
\begin{equation}\label{bnd}
\Phi^+(t) - e^{i\theta(t)}\Phi^-(t) = 2i\sin \theta(t) e^{-t} J \Phi(-t), \quad t\in \Real_+,
\end{equation}
where $J=\left(\begin{smallmatrix}
0 & 1\\
1 & 0
\end{smallmatrix}\right)
$
and $\theta(t)=\arg\{\Lambda^+(t)\}$ with the $\arg\{\cdot\}$ branch chosen so that $\theta(t)$ is continuous and vanishes as $t\to\infty$, that is
\begin{equation}\label{thetat}
\theta(t)  =  
\arctan \frac
{  
\sin\frac {1-\alpha}2\pi  
}
{
  \kappa_\alpha^{-1} \nu^{1-\alpha}  \Big((t/\nu)^{3-\alpha}   +  (t/\nu)^{1-\alpha}\Big) +   (t/\nu)^{3-\alpha}  +  \cos \frac {1- \alpha}2\pi  
},\quad t> 0.
\end{equation} 
Now, by Lemma \ref{lem:3.1}, the problem \eqref{eigpr} reduces to finding sectionally holomorphic functions $(\Phi_0, \Phi_1)$,
growing as in \eqref{near_zero_est}-\eqref{near_inf_est}, which comply with constraint \eqref{poles}
and satisfy the boundary condition \eqref{bnd}.

In general such Riemann problem for a pair of functions may not have an explicit solution, but the system \eqref{bnd} can be  decoupled,  see  \cite[eq. (5.35)]{ChK} and consequently $\Phi_0(z)$ and $\Phi_1(z)$,
satisfying the conditions \eqref{near_zero_est}-\eqref{near_inf_est}, can be expressed in terms of 
solutions to certain integral equations, see \eqref{pqeq} below. 
More precisely, define sectionally holomorphic  function 
\begin{equation}\label{Xzdef}
X(z) = \exp \left(\frac 1 \pi \int_0^\infty \frac{\theta(\tau)}{\tau-z}d\tau\right),
\end{equation}
and the real valued function 
$$
h(t)  = e^{i\theta(t)}\sin\theta(t) X(-t)/X^+(t), \quad t\in \Real_+.
$$

\begin{lem}
\

\medskip
1. For all  $\nu>0$ large enough,  the integral equations 
\begin{equation}\label{pqeq}
\begin{aligned}
q_\pm(t) = & \pm \frac 1{ \pi } \int_0^\infty \frac{   h(\nu \tau) e^{-\nu \tau}}{\tau+t}q_\pm(\tau)d\tau +t \\
p_\pm(t) = & \pm  \frac {1}{ \pi  } \int_0^\infty \frac{   h(\nu \tau) e^{- \nu \tau}}{\tau +t }p_\pm(\tau)d\tau    +1 
\end{aligned} \qquad t>0,
\end{equation}
have unique solution such that $p_\pm (t)-1$ and $q_\pm (t)-t$ are square integrable on $\Real_+$.

\medskip 
2. The solutions to the Riemann problem with boundary conditions \eqref{bnd} and growth as in \eqref{near_zero_est}-\eqref{near_inf_est}
has the form 
\begin{equation}\label{Phi}
\Phi(z) = X(z) A(z/\nu) \begin{pmatrix}
C_1 \\
C_2
\end{pmatrix}
\end{equation}
where
\begin{equation}\label{apmbpm}
A(z) = \begin{pmatrix}
-a_-(z)  &b_\alpha(\nu) a_+(z) -\nu b_+ (z) \\
-a_+(z) & b_\alpha(\nu) a_-(z)-\nu b_-(z)
\end{pmatrix}\quad \text{with} \quad
\begin{aligned}
a_\pm (z) = & p_+(z)\pm p_-(z)\\
b_\pm (z) = &q_+(z)\pm q_-(z)
\end{aligned}
\end{equation}
and 
\begin{equation}\label{beta_alpha}
b_\alpha(\nu) := \frac 1 \pi \int_0^\infty \theta(\nu\tau)d\tau.
\end{equation}
\end{lem}

\medskip
This lemma is checked exactly as in \cite{ChK}, see calculations preceding (5.39) therein.
At this stage, the eigenproblem \eqref{eigpr} is equivalent to finding all $\nu>0$, for which there exists a nonzero vector 
$\left(\begin{smallmatrix}
C_1\\
C_2
\end{smallmatrix}\right)
$
satisfying \eqref{poles} with $\Phi$ defined by \eqref{Phi}. Obviously, the condition \eqref{poles} reads 
\begin{multline*}
 e^{i\nu/2} X(i\nu) \Big(-a_-(-i)C_1 + \big(b_\alpha(\nu) a_+(-i)-b_+(-i)\big)C_2\Big)+ \\
 e^{-i\nu/2}X(-i\nu) \Big(-a_+(i) C_1 + \big(b_\alpha(\nu) a_-(i)-b_-(i)\big)C_2\Big)=0,
\end{multline*}
or, after a rearrangement, $\eta(\nu) C_1 + \xi(\nu) C_2 =0$, where 
\begin{equation}\label{xi_eta}
\begin{aligned}
\xi(\nu) :=\; &   e^{i\nu/2}  X(\nu i)\Big(  b_+(-i)  - b_\alpha(\nu) a_+(-i) \Big)  
+ e^{-i\nu/2}X(-\nu i)\Big( b_-(i)-  b_\alpha(\nu) a_-(i) \Big)   \\
\eta(\nu) :=\; & e^{i\nu/2} X(\nu i) a_-(-i)+ e^{-i\nu/2} X(-\nu i) a_+(i).  
\end{aligned}
\end{equation}
Since $C_1$ and $C_2$ are real valued, nontrivial solutions are possible if and only if 
\begin{equation}\label{alg}
 \eta(\nu)\overline{\xi(\nu)}-\overline{\eta(\nu)} \xi(\nu) =0.
\end{equation}

To recap,  any solution $(\nu,\, p_\pm, q_\pm)$ to the system of algebraic and integral equations 
\eqref{alg} and \eqref{pqeq}, can be used to construct a triplet $(\nu, \Phi_0,\Phi_1)$, mentioned above, and 
consequently, a solution $(\lambda, \varphi)$  to the eigenproblem \eqref{eigpr}. 
In particular, all eigenvalues can be found  by solving this equivalent problem. 

\subsection{Asymptotic analysis}
While the equivalent problem, presented in the previous section, does not appear simpler than the initial eigenproblem, 
it happens to be more accessible asymptotically as $\nu\to \infty$. For all $\nu$ large enough, the system of 
equations \eqref{lambdanu} and \eqref{pqeq} is solved by the fixed point iterations. This in turn implies that 
it has countably many solutions, which under a suitable enumeration admit asymptotically exact approximation as the enumeration 
index tends to infinity.  

Let us describe this approximation in greater detail.
It can be shown as in \cite[Lemma 5.6]{ChK} that the integral operator in the right hand side of \eqref{pqeq} is a 
contraction on $L_2(\Real_+)$, at least for all $\nu$ large enough. It also follows that 
the functions in \eqref{apmbpm} satisfy the estimates 
\begin{equation}\label{abpmest}
\begin{aligned}
&
 |a_+(\pm i) -2|\vee |a_-(\pm i)|\le C /\nu   \\  
&
 |b_+(\pm i) \mp 2i| \vee |b_-(\pm i)| \le C / {\nu^{2}} 
\end{aligned}
\end{equation}
with a constant $C$, which depends only on $\alpha$, in a certain uniform way as in Lemma \ref{lem2.5} below
(see also \cite[Lemm 5.7]{ChK}).

To proceed we need the following additional estimates, specific to mixed covariance \eqref{Ktilde}: 

\begin{lem}\label{lem2.5}
\

\medskip 
\noindent
a) 
For any $\alpha_0 \in (0,1)$, there are constants $C_0$ and $\nu_0$ such that for all $\nu\ge \nu_0$
and $\alpha\in [\alpha_0, 1]$ the functions defined in \eqref{Xzdef} and \eqref{beta_alpha} satisfy the bounds 
$$
\Big|\arg\{X(\nu i)\}\Big|\le (1-\alpha)C_0 \nu^{\alpha-1}\quad \text{and}\quad \Big| |X(\nu i) |-1\Big| \le (1-\alpha)C_0 \nu^{\alpha-1}
$$
and 
$$
b_\alpha(\nu) \le (1-\alpha)C_0\nu^{\alpha-1}.
$$

\medskip 
\noindent
b)  For any $\alpha_0 \in (1,2)$,  there are constants $C_0$ and $\nu_0$ such that for all $\nu\ge \nu_0$
and $\alpha\in [1, \alpha_0]$ the functions defined in \eqref{Xzdef} and \eqref{beta_alpha} satisfy the bounds
\begin{equation}\label{Xcargabs}
\Big|\arg\{X_c(\nu i)\} - \tfrac{1-\alpha}{8}\pi\Big| \le C_0 \nu^{1-\alpha}\quad \text{and} \quad 
\Big| |X_c(\nu i)|- \sqrt{\tfrac {3-\alpha} 2}\Big| \le C_0 \nu^{1-\alpha}
\end{equation}
and 
\begin{equation}\label{theta_m1}
\left|b_\alpha(\nu) - b_\alpha\right|\le C_0 \nu^{1-\alpha}\quad \text{with\ } b_\alpha = \frac{\sin \Big(\frac \pi {3-\alpha}\frac {1-\alpha} 2\Big)}{\sin \frac \pi {3-\alpha}}.
\end{equation}

\end{lem}

\begin{proof}\

\medskip 

a) 
Define $\Gamma(z) := \frac 1\pi \int_0^\infty \frac{\theta(\tau)}{\tau-z}d\tau$, then using the expression in \eqref{thetat}, 
positive for $\alpha\in (0,1)$, 
\begin{align*}
\big|\arg\{X(\nu i)\}\big| = 
& \big|\Im\{\Gamma(\nu i)\}\big| =\frac 1 \pi 
 \int_0^\infty \frac{\theta(\tau)}{\tau^2+\nu^2}\nu d\tau =
 \frac 1 \pi  \int_0^\infty \frac{\theta(\nu s)}{s^2+1}  ds  
\le \\
&
%\int_0^\infty 
%\frac 1 {s^2+1} 
%\arctan \frac
%{  
%\sin\frac {1-\alpha}2\pi  
%}
%{
%  \kappa_\alpha^{-1} \nu^{1-\alpha}  \big(s^{3-\alpha}   +  s^{1-\alpha}\big) +   s^{3-\alpha}  +  \cos \frac {1- \alpha}2\pi  
%}
%ds \le 
\frac {\sin \frac {1-\alpha}2\pi } {\kappa_\alpha^{-1} \nu^{1-\alpha}}
\int_0^\infty 
\frac 1 {s^2+1} 
 \frac
{  
1  
}
{
    \big(s^{3-\alpha}   +  s^{1-\alpha}\big)   
}
ds \le (1-\alpha)C_0 \nu^{\alpha-1}. 
\end{align*}
Similarly 
$$
\Re \{\Gamma(\nu i)\}  =  \frac 1\pi \int_0^\infty \frac{\tau}{\tau^2+\nu^2} \theta(\tau) d\tau \le (1-\alpha) C_0 \nu^{\alpha-1}
$$ 
and hence for all $\nu$ large enough, $\big||X(\nu i)|-1\big| = \big|e^{\Re \{\Gamma(i\nu)\}}-1\big|\le  (1-\alpha)C_0 \nu^{\alpha-1}$ as claimed. 
The bound for $b_\alpha$ is obtained  similarly 
$$
b_\alpha(\nu) =    \frac 1 \pi \int_0^\infty \theta(\nu \tau)d\tau 
%=
%\frac 1 \pi  \int_0^\infty 
%\arctan \frac
%{  
%\sin\frac {1-\alpha}2\pi  
%}
%{
%  \kappa_\alpha^{-1} \nu^{1-\alpha}  \big(\tau^{3-\alpha}   +  \tau^{1-\alpha}\big) +   \tau^{3-\alpha}  +  \cos \frac {1- \alpha}2\pi  
%}
%d\tau 
\le  \frac {\sin \frac {1-\alpha}2} {\kappa_\alpha^{-1} \nu^{1-\alpha}}
  \int_0^\infty 
  \frac
{  
1 
}
{
    \big(\tau^{3-\alpha}   +  \tau^{1-\alpha}\big)   
}
d\tau \le (1-\alpha)C_0 \nu^{\alpha-1},
$$
where the last inequality holds since $\min_{0\le \alpha\le 1}\kappa_\alpha^{-1} > 0$ and the integral is bounded uniformly over $\alpha \in [\alpha_0, 1]$ for all $\alpha_0\in (0,1) $. 

\medskip 
b) 
The estimate \eqref{theta_m1} holds since (see  \cite[Lemma 5.2]{ChK})
$$
b_\alpha = 
\frac 1 \pi \int_0^\infty
\arctan \frac{  
\tau^{\alpha-3}  \sin \frac {1- \alpha}2\pi   
}
{
1 +  \tau^{\alpha-3} \cos \frac {1- \alpha}2\pi 
}
d\tau 
$$
and therefore    
\begin{equation}\label{int_theta_bnd}
\left|\frac 1\pi \int_0^\infty \theta(\nu \tau )d\tau - b_\alpha \right| \le   
%\int_0^\infty 
%\left|
%\frac{  
%\tau^{\alpha-3}  \sin \frac {1- \alpha}2\pi   
%}
%{
%\kappa_\alpha^{-1}\nu^{1-\alpha} \big(1 + \tau^{-2} \big) +   1 +  \tau^{\alpha-3} \cos \frac {1- \alpha}2\pi 
%}
%- 
%\frac{  
%\tau^{\alpha-3}  \sin \frac {1- \alpha}2\pi   
%}
%{
%1 +  \tau^{\alpha-3} \cos \frac {1- \alpha}2\pi 
%}
%\right|
%d\tau \le \\
%&
\nu^{1-\alpha}\kappa_\alpha^{-1}\sin \tfrac {1- \alpha}2\pi
\int_0^\infty
\frac
{
 \big(1 + \tau^{-2} \big) \tau^{\alpha-3}  
}
{
1+\tau^{2\alpha-6}   
}
d\tau 
\end{equation}
where we used the identity $\arctan x -\arctan y = \arctan (x-y)/(1+xy)$.

Further, define  
$$
\theta_0(\tau) := \arctan 
\frac{  
\tau^{\alpha-3}  \sin \frac {1- \alpha}2\pi   
}
{
  1 +  \tau^{\alpha-3} \cos \frac {1- \alpha}2\pi 
},\qquad \tau>0
$$
and  
$$
X_0(i) := \exp \left(\frac 1 \pi \int_0^\infty \frac{\theta_0(\tau)}{\tau - i}d\tau \right).
$$
It is shown in  \cite[Lemma 5.5]{ChK} that 
$$
|X_0(i)| = \sqrt{\frac {3-\alpha} 2} \quad \text{and} \quad 
\arg\{X_0(i)\} = \frac{1-\alpha}{8}\pi.
$$
The estimates in \eqref{Xcargabs} are obtained by bounding 
$
|\theta(\nu \tau)-\theta_0(\tau)| 
$
as in \eqref{int_theta_bnd}. 

\end{proof}
 
These estimates are the key to the following approximation: 

\begin{lem} 
The solutions to the system \eqref{pqeq} and \eqref{alg} can be enumerated so that 
\begin{equation}\label{nun2}
\nu_n = \pi n -\frac \pi 2 + g(\nu_n) +  n^{-1} r_n(\alpha), \quad n\in \mathbb{Z}
\end{equation}
where $\sup_{\alpha \in [\alpha_0, 1]}r_n(\alpha)<\infty$ for any $\alpha_0 \in (0,1)$
and 
\begin{equation}\label{gnu}
g(\nu) := -   2\arg\{X(\nu i)\} + \frac \pi 2 - \arg\{    i  + b_\alpha(\nu) \}.
\end{equation}
\end{lem}

\begin{proof}
Plugging the estimates \eqref{abpmest} into expressions from \eqref{xi_eta} we obtain
\begin{equation}\label{seeme}
 \xi(\nu)\overline{\eta(\nu)}=\,   
% e^{i\nu }  X(\nu i)^2 \big(  b_+(-i)  - b_\alpha(\nu) a_+(-i) \big) a_+(-i)\Big(1+R(\nu)\Big) =\\
%&
- 4 e^{ \nu i}  X(\nu i)^2 \big(   i  + b_\alpha(\nu) \big)  \Big(1+R_1(\nu)\Big) =
-  4 e^{ \nu i +\frac \pi 2 i }    \Big(1+R_2(\nu)\Big) 
\end{equation}
where $|R_1(\nu)|\le C \nu^{-1}$ and $|R_2(\nu)|\le C \nu^{\alpha-1}$ with a constant $C$ which depends only on $\alpha_0$.  
Hence equation \eqref{alg} with $\nu>0$ reads
$$
\nu +\frac \pi 2 + \arctan \frac{\Im\{R_2(\nu)\}}{1+\Re\{R_2(\nu)\}}= \pi n, \qquad n\in \mathbb{Z}.
$$
A tedious but otherwise straightforward calculation reveals that $|R_2'(\nu)|\le C \nu^{\alpha-1}$ as well.
Hence for each $n$ large enough,  the unique solution to the system \eqref{lambdanu} and \eqref{pqeq} 
is obtained by fixed-point iterations. 
The algebraic component $\nu_n$ of this solution satisfies, cf. \eqref{seeme} and  \eqref{alg}
$$
\nu_n  +   2\arg\{X(\nu_n i)\} + \arg\{    i  + b_\alpha(\nu_n) \} + \arctan \frac{\Im\{ R_1(\nu_n)\} }{1+\Re\{ R_1(\nu_n)\}} =\pi n
$$
which yields the claimed formula since $|R_1(\nu)|\le C \nu^{-1}$.
\end{proof}

The enumeration of the eigenvalues, introduced by this lemma, may differ from that which puts them in the decreasing order. 
A calibration procedure, detailed in  \cite[Section 5.1.7]{ChK}, shows that the two enumerations do in fact coincide.  
The formulas in Theorem \ref{eig-thm} are obtained from \eqref{lambdanu} and \eqref{nun2}, 
by plugging the estimates from Lemma \ref{lem2.5} into \eqref{gnu} and replacing $\alpha = 2-2H$.

\section{Proof of Theorem \ref{main-thm}}

The proof uses the theory of small deviations developed in \cite{DLL96}, which addresses the problem of 
calculating exact asymptotics of the probabilities 
\begin{equation}\label{Pr}
\P\Big(\sum_{j=1}^\infty \phi(j) Z_j \le r\Big)\quad \text{as\ } r\to 0,
\end{equation}
where $Z_j$'s are i.i.d. nonnegative random variables and $\phi(j)$ is a summable sequence of positive numbers. 
Squared $L_2$-norm of a Gaussian process can be written as such series with $Z_j\sim \chi^2_1$  and $\phi(j):=\lambda_j$, where 
$\lambda_j$'s are the eigenvalues of its covariance operator.
In what follows $\phi(t)$, $t\in \Real_+$ stands for the function, obtained by replacing the integer
index in $\phi(j)$ with a real positive variable $t$. 

The main ingredients in the asymptotic analysis of \eqref{Pr} in \cite{DLL96} are functions, defined in terms of the Laplace transform 
$f(s):= \E e^{-s Z_1}=(1-2s)^{-\frac 1 2}$, $s\in (-\infty,\frac 1 2)$ of the $\chi^2_1$-distribution: 
\begin{equation}\label{I012}
\begin{aligned}
& I_0(u) :=  \int_1^\infty \log f(u \phi(t))dt =
-\frac 1 2 \int_1^\infty \log   (1+2u \phi(t))dt \\
& 
I_1(u) :=  \int_1^\infty u\phi(t) (\log f)'(u\phi(t))dt =
-\int_1^\infty  \frac {u\phi(t)} {1+2u\phi(t)}dt
\\
& 
I_2(u) := \int_1^\infty (u\phi(t))^2 (\log f)''(u\phi(t))dt = 2\int_1^\infty  \left(\frac {  u\phi(t) }{ 1+2u\phi(t) }\right)^2dt
\end{aligned}
\end{equation}
We will apply the following result: 
 
\begin{cor}[Corollary 3.2 from \cite{DLL96}]\label{cor} The $L_2$-ball probabilities satisfy 
\begin{equation}\label{Pfla}
\P\big(\|\widetilde B\|_2^2  \le r\big) \sim   \Big(\sqrt{u(r)}I_2(u(r))\Big)^{-\frac 12}\exp\Big(I_0(u(r))+ u(r) r\Big)\quad \text{as\ }r\to 0, 
\end{equation}
where $u(r)$ is any function satisfying 
\begin{equation}\label{cond}
\lim_{r\to 0}\frac{I_1(u(r))+u(r) r}{\sqrt{I_2(u(r))}}=0.
\end{equation} 
\end{cor}

\subsection{Asymptotic expansion of $I_j(u)$'s}
A preliminary step towards application of Corollary \ref{cor} is to derive the exact asymptotics of the functions from \eqref{I012} 
as $u\to\infty$ for the weight function
\begin{equation}\label{phit}
\phi(t) = \sum_{j=1}^k c_j t^{-d_j}.
\end{equation}
In our case, $k=3$ and $c_j$ and $d_1<d_2<d_3$ are positive constants, whose values depend on $H$.
In particular, 
$
\lim_{t\to\infty}\frac d{dt} \log \phi(t) =0
$
holds, as required in \cite{DLL96}. 
\nocite{NN04ptrf} 
It will be convenient to use constants 
$$
a_j   :=  c_j/c_1 \quad \text{and}\quad 
\delta_j := d_j - d_1 
$$
and to define the new variable $v$ by the formula 
$$
2u c_1 = v^{-d_1},
$$
which converges to zero as $u\to\infty$. Obviously $a_1=1$ and $\delta_1=0$, and for the specific values of constants $d_j$'s needed below, 
we also have $\delta_3=1$.

\subsubsection{Asymptotic expansion of $I_0(u)$} 
Integrating by parts we get
\begin{align*}
I_0(u) = &
-\frac 1 2 \int_1^\infty \log   (1+2u \phi(t))dt = \\
&
\frac 1 2\log   (1+2u \phi(1))  + \int_1^\infty \frac { u t\phi'(t)}{1+2u \phi(t)}  dt =\\
&
\frac 1 2\log   \Big(1+2u \sum_{j=1}^k c_j  \Big)  - \sum_{i=1}^k c_i d_i \int_1^\infty \frac {   u t^{-d_i} }{1+2u \sum_{j=1}^k c_j t^{-d_j}}  dt.
\end{align*}
Changing the integration variable and using the above notations, this can be written as
\begin{equation}\label{I0}
I_0(u) = \frac 1 2\log   \Big(1+v^{-d_1} \sum_{j=1}^k a_j  \Big)  
- \frac 1 2 \sum_{i=1}^k   a_i d_i v^{\delta_i-1}\int_v^\infty \frac { \tau^{-\delta_i} }{\tau^{d_1}+1 + p(v/\tau)}  d\tau,
\end{equation}
where we defined
$$
p(s) :=  a_2 s^{\delta_2}+a_3 s^{\delta_3}.
$$
Let us find the exact asymptotics of each integral as $v\to 0$. The first one gives
\begin{align*}
&
v^{\delta_1-1}\int_v^\infty \frac { \tau^{-\delta_1} }{\tau^{d_1}+1 + p(v/\tau)}  d\tau=
v^{-1}\int_v^\infty \frac { 1 }{\tau^{d_1}+1 + p(v/\tau)}  d\tau = \\
&
v^{-1}\int_v^\infty \frac { 1 }{\tau^{d_1}+1 }  d\tau
-
v^{-1}\int_v^\infty \frac { p(v/\tau) }{\big(\tau^{d_1}+1 + p(v/\tau)\big)\big(\tau^{d_1}+1\big)} d\tau =\\
&
v^{-1}\int_v^\infty \frac { 1 }{\tau^{d_1}+1 }  d\tau
-J_{0,1}(v)
-
J_{0,2}(v)
\end{align*}
where we defined 
\begin{align*}
J_{0,1}(v) & :=   
v^{\delta_2-1}a_2\int_v^\infty \frac {  \tau^{-\delta_2} }{\big(\tau^{d_1}+1 + p(v/\tau)\big)\big(\tau^{d_1}+1\big)} d\tau \\
J_{0,2}(v) &:=v^{\delta_3-1}a_3\int_v^\infty \frac {   \tau^{-\delta_3} }{\big(\tau^{d_1}+1 + p(v/\tau)\big)\big(\tau^{d_1}+1\big)} d\tau.
\end{align*}
The latter term with $\delta_3=1$ satisfies 
\begin{align*}
J_{0,2}(v) = &\, a_3\int_v^\infty \frac {   \tau^{-1} }{\big(\tau^{d_1}+1 + p(v/\tau)\big)\big(\tau^{d_1}+1\big)} d\tau =\\
%&
%a_3\int_v^\infty \frac {   \tau^{-1} }{\big(\tau^{d_1}+1\big)^2 } d\tau
%-
%a_3 v^{\delta_2}\int_v^\infty \frac{\tau^{-1-\delta_2}}{(\tau^{d_1}+1)^2} 
%\frac {  a_2  }{\tau^{d_1}+1 + p(v/\tau) }   d\tau 
%-
%a_3v \int_v^\infty \frac{\tau^{-1}}{(\tau^{d_1}+1)^2} 
%\frac { a_3  \tau^{-1} }{\tau^{d_1}+1 + p(v/\tau) }   d\tau =\\
&
a_3\int_v^\infty \frac {   \tau^{-1} }{\big(\tau^{d_1}+1\big)^2 } d\tau
+O(1) = -a_3\log v
+O(1),
\end{align*}
and similarly
\begin{align*}
J_{0,1}(v) =  &\,
v^{\delta_2-1}a_2\int_0^\infty \frac {  \tau^{-\delta_2} }{\big(\tau^{d_1}+1\big)^2} d\tau \\
&
-
v^{2\delta_2-1}a_2\int_v^\infty \frac {  \tau^{-2\delta_2} }{\big(\tau^{d_1}+1\big)^2}\frac {a_2  }{\tau^{d_1}+1 + p(v/\tau)}d\tau 
%-
%v^{\delta_2+\delta_3-1}a_2\int_v^\infty \frac {  \tau^{-\delta_2-\delta_3} }{\big(\tau^{d_1}+1\big)}\frac {a_3   }{\big(\tau^{d_1}+1 + p(v/\tau)\big)\big(\tau^{d_1}+1\big)}d\tau
+O(1).
\end{align*}
If $2\delta_2-1>0$, the second term on the right is of order $O(1)$, otherwise we can proceed similarly to obtain the expansion
$$
J_{0,1}(v) = -\sum_{k=1}^{m} (-a_2)^{k}\chi_{1,k} v^{k \delta_2-1 }
$$
where $m$ is the largest integer such that $m \delta_2-1 <0$ and, \cite[formula 3.241.4]{GR7ed2007},
\begin{equation}\label{chi1k}
\chi_{1,k} := \int_0^\infty \frac {  \tau^{-k\delta_2} }{\big(\tau^{d_1}+1\big)^{k+1}} d\tau= 
\frac 1 {d_1} \frac{\Gamma\big(\frac {1-k\delta_2} {d_1}\big)\Gamma\big(k+1-\frac{1-k\delta_2}{d_1}\big)}{\Gamma\big(k+1\big)},
\qquad k =0,...,m.
\end{equation}
% General formula from GR:
%$$
%\int_0^\infty \frac{\tau^{-\mu}}{(\tau^{d_1}+1)^m}d\tau=
%\frac 1 {d_1} \frac{\Gamma\big(\frac {1-\mu} {d_1}\big)\Gamma\big(m-\frac{1-\mu}{d_1}\big)}{\Gamma\big(m\big)}
%$$
Plugging all the estimates back we obtain 
\begin{equation}\label{int1}
v^{\delta_1-1}\int_v^\infty \frac { \tau^{-\delta_1} }{\tau^{d_1}+1 + p(v/\tau)}  d\tau
=  a_3\log v
+\sum_{k=0}^{m}  (-a_2)^{k}\chi_{1,k}v^{k \delta_2-1 }
 + O(1), \quad \text{as\ }v\to 0.
\end{equation}

Further, the second integral in the sum in \eqref{I0} reads
\begin{align*}
&
v^{\delta_2-1}\int_v^\infty \frac { \tau^{-\delta_2} }{\tau^{d_1}+1 + p(v/\tau)}  d\tau = 
v^{\delta_2-1}\int_v^\infty \frac { \tau^{-\delta_2} }{\tau^{d_1}+1 }  d\tau
- J_{1,1}(v)
- J_{1,2}(v).
\end{align*}
Here 
\begin{align*}
J_{1,2}(v):= & 
v^{\delta_2+\delta_3-1} a_3 \int_v^\infty   
\frac{   \tau^{-\delta_2-\delta_3} }{\big(\tau^{d_1}+1 + p(v/\tau)\big)\big(\tau^{d_1}+1\big)}   
d\tau \le  \\
&
v^{\delta_2+\delta_3-1} a_3 \int_v^\infty   
\frac{   \tau^{-\delta_2-\delta_3} }{\big(\tau^{d_1}+1  \big)^2 }   
d\tau = O(1)
\end{align*}
since $ \delta_2+\delta_3 >1$. The second term  
$$
J_{1,1}(v):=  v^{2\delta_2-1}a_2\int_v^\infty   
\frac{   \tau^{-2\delta_2}  }{\big(\tau^{d_1}+1 + p(v/\tau)\big)\big(\tau^{d_1}+1\big)}   
d\tau
$$
is of order $O(1)$, if $2\delta_2>1$. Otherwise, 
$$
J_{1,1}(v) = v^{2\delta_2-1}a_2\int_0^\infty   
\frac{   \tau^{-2\delta_2}  }{\big(\tau^{d_1}+1\big)^2}   
d\tau
- v^{3\delta_2-1}a_2^2\int_v^\infty  \frac{\tau^{-3\delta_2}}{\big(\tau^{d_1}+1\big)^2} 
\frac{   1  }{\tau^{d_1}+1 + p(v/\tau)}   
d\tau + O(1)
$$
where the second term is of order $O(1)$, if $3\delta_2>1$ and so on. Thus we obtain asymptotics  
\begin{equation}\label{int2}
v^{\delta_2-1}\int_v^\infty \frac { \tau^{-\delta_2} }{\tau^{d_1}+1 + p(v/\tau)}  d\tau 
=
\sum_{k=1}^{m} (-a_2)^{k-1}\chi_{0,k}v^{k\delta_2-1}  + O(1), \quad\text{as\ } v\to 0
\end{equation}
with 
\begin{equation}\label{chi2k}
\chi_{0,k} : = \int_0^\infty  \frac{\tau^{-k\delta_2}}{\big(\tau^{d_1}+1\big)^k} d\tau =
\frac 1 {d_1} \frac{\Gamma\big(\frac {1-k\delta_2} {d_1}\big)\Gamma\big(k-\frac{1-k\delta_2}{d_1}\big)}{\Gamma\big(k\big)}
,\qquad k=1,...,m.
\end{equation}
Finally, since $\delta_3=1$, the last summand in \eqref{I0} contributes  
\begin{equation}\label{int3}
v^{\delta_3-1}\int_v^\infty \frac { \tau^{-\delta_3} }{\tau^{d_1}+1 + p(v/\tau)}  d\tau =  
\int_v^\infty \frac { \tau^{-1} }{\tau^{d_1}+1 }  d\tau  
-
J_{3,1}(v)
-
J_{3,2}(v)=
-\log v 
+O(1)
\end{equation}
where we used the estimates
$$
J_{3,1}(v) :=v^{\delta_2+\delta_3-1}\int_v^\infty   \frac {  a_2  \tau^{-\delta_3-\delta_2}  }{\big(\tau^{d_1}+1 + p(v/\tau)\big)\big(\tau^{d_1}+1\big)}   
d\tau 
%\le v^{\delta_2+\delta_3-1}\int_v^\infty   \frac {  a_2  \tau^{-\delta_3-\delta_2}  }{ \big(\tau^{d_1}+1\big)^2}   d\tau 
= O(1)
$$
and
$$
J_{3,2}(v) :=
v^{2\delta_3-1}\int_v^\infty  \frac {   a_3 \tau^{-2\delta_3} }{\big(\tau^{d_1}+1 + p(v/\tau)\big)\big(\tau^{d_1}+1\big)}   
d\tau
%\le v^{2\delta_3-1}\int_v^\infty  \frac {   a_3 \tau^{-2\delta_3} }{\big(\tau^{d_1}+1\big)^2}   d\tau
= O(1).
$$
Plugging \eqref{int1}, \eqref{int2} and \eqref{int3} into \eqref{I0} we obtain 
\begin{equation}\label{I0v}
\begin{aligned}
I_0(u) = &
-\tfrac 1 2d_1\big( 1  +     a_1  a_3 -     a_3 \tfrac{d_3}{d_1}  \big)\log v \\
&
-\frac 1 2    a_1 d_1    \chi_{1,0}v^{-1 }
+ \frac 1 2        \sum_{k=1}^{m} (-a_2)^{k }\Big(d_2  \chi_{0,k}-a_1 d_1 \chi_{1,k}\Big)v^{k\delta_2-1}
+ O(1).
\end{aligned}
\end{equation}

\subsubsection{Asymptotic expansion of $I_1(u)$}
Using the asymptotic formulas, already derived above, we have
\begin{align}
& I_1(u)  :=   \nonumber
-\sum_{i=1}^k c_i\int_1^\infty  \frac {u t^{-d_i}} {1+2u\sum_{j=1}^k c_j t^{-d_j}}dt = 
- \frac 1 2\sum_{i=1}^k a_i v^{\delta_i-1}\int_v^\infty \frac { \tau^{-\delta_i} }{\tau^{d_1}+1 + p(v/\tau)}  d\tau =\\
&\label{I1lim}
\frac 1 2 a_3 ( 1-a_1)    \log v
- \frac 1 2  a_1    \chi_{1,0}v^{-1 }
+ \frac 1 2   \sum_{k=1}^{m} (-a_2)^{k }\big(\chi_{0,k}-a_1\chi_{1,k}\big)v^{k\delta_2-1} 
+ O(1),
\end{align}
where the logarithmic term vanishes since $a_1=1$.

\subsubsection{Asymptotic expansion of $I_2(u)$}
We will need only the leading asymptotic term of $I_2(u)$. To this end, we have 
\begin{equation}\label{I2lim}
\begin{aligned}
I_2(u) =\, &  2\int_1^\infty  \left(\frac {  u \sum_{j=1}^k c_j t^{-d_j} }{ 1+2u \sum_{j=1}^k c_j t^{-d_j} }\right)^2dt =  \\
&
\frac 1 2 v^{-1} \int_v^\infty  \left(
\frac{1+p(v/\tau) }{ \tau^{d_1}+ 1+ p(v/\tau) }
\right)^2d\tau = 
\frac 1 2\chi_{3,1} v^{-1}\big(1+o(1)\big)
\end{aligned}
\end{equation}
where we defined 
\begin{equation}\label{chi31}
\chi_{3,1}:=  \int_0^\infty  \frac{1  }{ \big(\tau^{d_1}+ 1\big)^2} d\tau = \frac 1 {d_1} \frac{\Gamma\big(\frac {1} {d_1}\big)\Gamma\big(2-\frac{1}{d_1}\big)}{\Gamma\big(2\big)}.
\end{equation}

\subsection{Case $\boldsymbol{H \in(\frac 1 2, 1)}$}  
As in the previous section,  $\alpha:=2-2H\in (0,1)$. 
By Theorem \ref{eig-thm}, eigenvalues of the covariance operator of $\widetilde B$ satisfy 
$$
\lambda(n) =    \nu_n^{-2}      +    \kappa_\alpha   \nu_n^{ \alpha-3} 
$$
where $\kappa_\alpha$ is the constant from \eqref{mixed_lambdan} expressed in terms of $\alpha$ (see \eqref{lambdanu} below):
$$
\kappa_\alpha:=\dfrac {c_\alpha} {\Gamma(\alpha)}  \dfrac {\pi} {\cos \frac \pi 2\alpha}=\Gamma(2H+1)\sin (\pi H),
$$
and, for $\alpha\in (0,1)$,   $\nu_n = (n-\frac 1 2)\pi+O(n^{\alpha-1})$ as $n\to\infty$.
Taylor expansion yields 
\begin{align*}
\nu_n^{-2} =\, & %(\pi n)^{-2}\Big(1-\frac 1  {2n} +O(n^{\alpha-2})\Big)^{-2} =
% (\pi n)^{-2}\Big(1+\frac 1  { n} +O(n^{\alpha-2})\Big) =
\frac 1 {\pi^2} n^{-2}+\frac 1 {\pi^2} n^{-3}  +O(n^{\alpha-4}) 
\\
\nu_n^{\alpha-3} =\, &  
%(\pi n)^{\alpha-3}\Big(1-\frac 1 {2n} +O(n^{\alpha-2})\Big)^{\alpha-3}=
%(\pi n)^{\alpha-3}\Big(1+(3-\alpha)\frac 1 {2n} +O(n^{\alpha-2})\Big) =
\frac 1 {\pi^{3-\alpha}} n^{\alpha-3}+\frac{3-\alpha} 2 \frac 1 {\pi^{3-\alpha}} n^{\alpha-4}     +O(n^{2\alpha-5}) 
\end{align*}
and consequently 
$$
\lambda(n) =    \frac 1 {\pi^2} n^{-2}+\frac {\kappa_\alpha} {\pi^{3-\alpha}} n^{\alpha-3}+\frac 1 {\pi^2} n^{-3}  +O(n^{\alpha-4}) =: 
\phi(n) + O(n^{\alpha-4}).
$$
By Li's comparison theorem, \cite[Theorem 2]{Li92},
$$
\P\Big(\sum_{n=1}^\infty \lambda(n) Z_n\le \eps^2\Big)\sim \Big(\prod_{n=1}^\infty \frac{\phi(n)}{\lambda(n)}\Big)^{1/2}
\P\Big(\sum_{n=1}^\infty \phi(n) Z_n\le \eps^2\Big)\quad \text{as\ } \eps\to 0
$$
if $\sum_{n=1}^\infty |1-\lambda(n)/\phi(n)|<\infty$, which holds in our case. Hence the desired asymptotics 
coincides with 
$$
\P\Big(\sum_{n=1}^\infty \phi(n) Z_n\le \eps^2\Big)\quad \text{as\ }\eps\to 0,
$$
up to a multiplicative constant. The function $\phi(t)$ has the form \eqref{phit} with  $k=3$ and  
\begin{equation}\label{csds}
\begin{aligned}
c_1 = \frac 1 {\pi^2}, &\qquad  d_1 = 2\\
c_2 =\frac {\kappa_\alpha} {\pi^{3-\alpha}},&\qquad d_2 = 3-\alpha\\
c_3 = \frac 1 {\pi^2}, & \qquad d_3 = 3
\end{aligned}
\end{equation}
Plugging these values into the asymptotic expansions \eqref{I0v}, \eqref{I1lim} and \eqref{I2lim} gives
\begin{align}
\label{I0limu}
I_0(u) & = 
 \frac 1 4\log  u  
-   \frac 1 {\sqrt 2   }u^{\frac1 2}
+  \sum_{k=1}^{m_\alpha}g_ku^{ \frac{1-k(1-\alpha)} 2}
+ O(1) \\
\label{I1limu}
I_1(u) & = -    \frac {1} {2\sqrt 2}  u^{\frac 1 2} -    \sum_{k=1}^{m_\alpha} h_k  u^{\frac{1-k(1-\alpha)}2 } + O(1) \\
\label{I2limu}
I_2(u) & = \sqrt{\frac {c_1} 2}\chi_{3,1}  u^{\frac 1 2} \big(1+o(1)\big)
\end{align}
where   $m_\alpha := \big\lfloor \frac 1 {1-\alpha}\big\rfloor$ and   
\begin{equation}\label{hkgk}
\begin{aligned}
%h_k := & \frac 1 2\big(-\tfrac {c_2}{c_1}\big)^{k }\big(\chi_{1,k}-\chi_{0,k}  \big)(2c_1)^{\frac{1-k(1-\alpha)}2 } \\
%g_k := & \big(-\tfrac {c_2}{c_1}\big)^{k }\big(\tfrac {d_2} 2  \chi_{0,k}-   \chi_{1,k}\big) 
%(2c_1)^{ \frac{1-k(1-\alpha)} 2}
h_k & :=    \frac 1 2(-\tfrac {c_2}{c_1})^{k }\big(\chi_{1,k}-\chi_{0,k}  \big)(2c_1)^{ \frac {1-k\delta_2}{d_1}} \\
g_k & := \frac 1 2(-\tfrac {c_2}{c_1})^{k }\Big(d_2  \chi_{0,k}-  d_1 \chi_{1,k}\Big)(2c_1)^{ \frac{1-k\delta_2 }{d_1}} 
\end{aligned}
\end{equation}
with constants $\chi_{i,k}$ defined in \eqref{chi1k}, \eqref{chi2k} and \eqref{chi31} and $\delta_2 = d_2-d_1$.

Application of Corollary \ref{cor} requires finding a function $u(r)$ which satisfies condition \eqref{cond}. 
To this end consider the equation (c.f. \eqref{I1limu}) 
\begin{equation}\label{eq_u}
 \frac {1}{2\sqrt 2}  u^{\frac 1 2} +  \sum_{k=1}^{\big\lfloor \frac{1}{2}\frac 1{1-\alpha}\big\rfloor} h_k u^{\frac{1-k(1-\alpha)}2 }
=ur,
\end{equation}
with respect to $u>0$. If we divide both sides by $u$, the left hand side becomes a monotonous function, which decreases to zero as $u\to\infty$,
and therefore this equation has unique positive solution $u(r)$, which grows to $+\infty$ as $r\to 0$. \
By the choice of the upper limit in the sum in \eqref{eq_u}, the power of $u(r)$ in the numerator of \eqref{cond} is strictly less than 
$\frac 1 4$ and hence \eqref{cond} holds in view of \eqref{I2limu}.

If we now let $u = (ry)^{-2}$ equation \eqref{eq_u} reads
$$
 \frac {1}{2\sqrt 2}  y +  \sum_{k=1}^{\left\lfloor \frac{1}{2}\frac 1{1-\alpha}\right\rfloor} h_k r^{ k(1-\alpha)  }y^{ k(1-\alpha)  +1 }
=1,\quad y>0.
$$ 
The function $r\mapsto y(r)$ is analytic in a vicinity of zero and can be expanded into series of the small parameter $r^{1-\alpha}$
$$
y(r) = y_0 + \sum_{j=1}^\infty y_j r^{j(1-\alpha)}.
$$ 
Let $\xi_j$ and $\eta_{k,j}$ be the coefficients of the expansions 
$$
 y(r)^{-2} =   \sum_{j=0}^{\infty} \xi_j r^{j(1-\alpha) } \quad \text{and}\quad 
 y(r)^{ k(1-\alpha)-1  } = \sum_{j=0}^\infty \eta_{k,j} r^{j(1-\alpha) }.
$$
Note that both are expressible in terms of $y_j$'s. Plugging these expansions into \eqref{Pfla} gives 
\begin{align*}
&
\P\big(\|\widetilde B\|_2^2  \le r\big) \sim \\
&
 r^{\frac 1 2}\exp\Big(
-   \frac 1 {\sqrt 2   }\sum_{j=0}^\infty \eta_{0,j} r^{j(1-\alpha)-1}
+  \sum_{k=1}^{m_\alpha} g_k
\sum_{j=0}^\infty \eta_{k,j} r^{(j+k)(1-\alpha)-1}
+ \sum_{j=0}^{\infty} \xi_j r^{j(1-\alpha)-1}  \Big) \sim \\
&
  r^{\frac 1 2}\exp\Big(
-   \frac 1 {8}   r^{-1} 
+ \sum_{\ell=1}^{m_\alpha} \Big(\xi_\ell    -   \frac 1 {\sqrt 2   }\eta_{0,\ell}\Big) r^{\ell(1-\alpha)-1}
+
\sum_{\ell=1}^{m_\alpha} \Big(\sum_{j=0}^{\ell-1} g_{\ell-j}\eta_{\ell-j,j} \Big)r^{\ell(1-\alpha)-1} 
  \Big).
\end{align*}
Changing back to $r=\eps^2$ and $\alpha=2-2H$, we obtain the following result, which implies \eqref{Peps} for $H\in (\frac 1 2, 1)$ 
and defines all of its ingredients: 

\begin{prop}\label{main-thm-largeH}

For $H\in (\frac 1 2,1)$ let $h_k$ and $g_k$ be the real sequences, defined by formulas \eqref{hkgk}
evaluated at the constants in \eqref{csds}.
Then for any $r>0$ the equation 
\begin{equation}\label{yreq}
y/y_0 + \sum_{k=1}^{\big\lfloor \frac 1 2\frac 1 {2H-1}\big\rfloor}   h_k r^{k(2H-1)}y^{k(2H-1)+1}=1
\end{equation}
with $y_0 = 2\sqrt{2}$ has unique positive root $y(r)$, which can be expaned into series  
\begin{equation}\label{yseries}
y(r) = y_0 + \sum_{j=1}^\infty y_j r^{j(2H-1)},
\end{equation}
convergent for all $r$ small enough.  Let $\xi_j$ and $\eta_{k,j}$ be coefficients of the power expansions  
\begin{equation}\label{xieta}
 y(r)^{-2} =   \sum_{j=0}^{\infty} \xi_j r^{j(2H-1) } \quad \text{and}\quad 
 y(r)^{ k(2H-1)-1  } = \sum_{j=0}^\infty \eta_{k,j} r^{j(2H-1) }
\end{equation}
and define 
$$
\beta_{\ell }(H) := 
\frac {1} {\sqrt 2   }  \eta_{0,\ell}  - \sum_{j=0}^{\ell-1} g_{\ell-j} \eta_{\ell-j,j} -\xi_\ell .
$$
Then (cf. \eqref{Peps})
\begin{equation}\label{Prob}
\P\big(\|\widetilde B\|_2   \le \eps\big)   \sim     \eps\exp\bigg(
-\frac 1 {8} \eps^{-2} -
    \sum_{\ell=1}^{\big\lfloor \frac 1 {2H-1}\big\rfloor}\beta_{\ell}(H) \eps^{2\ell(2H-1) -2} 
 \bigg), \quad \eps \to 0.
\end{equation}

\end{prop}

While the general closed form formula for constants $\beta_\ell$ in \eqref{Prob} would be cumbersome to derive,
they can be easily computed for any given value of $H$, at least numerically. The following example, demonstrates the algorithm 
of Proposition \ref{main-thm-largeH}. 

\begin{example}\

\medskip
\noindent
{\em Case $H\in [\frac 3 4,1)$.}  For the values of $H$ in this range,  $\lfloor \frac 1 {2H-1}\rfloor=1$ and the sum in 
\eqref{Prob} contains only one term:
$$
\P(\|\widetilde B\|_2\le \eps) \sim C(H) \eps \exp \Big(-\frac1 8 \eps^{-2}-\beta_1 \eps^{4H-4}\Big), \quad \eps\to 0,
$$
Thus we need to calculate only 
$
\beta_1 =  \frac 1 {\sqrt 2} \eta_{0,1}-g_1\eta_{1,0}-\xi_1.
$

Since $\lfloor \frac 1 2\frac 1 {2H-1}\rfloor=0$
equation \eqref{yreq} reduces to
$
y=y_0.
$
Comparing  
$
y(r)^{-2} = y_0^{-2} 
$
to \eqref{xieta} gives $\xi_1 = 0$. Similarly, comparing 
$y(r)^{-1  } =      y_0^{-1}$
with $ y(r)^{-1  }= \eta_{0,0}   +\eta_{0,1} r^{2H-1 } + O(r^{4H-2})$ gives $\eta_{0,1}=0$.
Finally $\eta_{1,0}=y_0^{2H-2}$ and hence 
$
\beta_1 = -g_1y_0^{2H-2}.
$
After simplification, formula \eqref{hkgk} yields
$
g_1 =  
% -2^{ \frac{ \alpha } 2}\tfrac {\kappa_\alpha}{\pi }  \big(\tfrac {3-\alpha} 2  \chi_{0,1}-   \chi_{1,1}\big)   =
% -2^{   -1-H }\tfrac {\sin (\pi H)}{\pi }    \Gamma(2H+1)\Gamma\big(1-H\big)\Gamma\big(H\big) =
- {2^{ -  H-1}}\Gamma(2H+1)       
$
%
% ==== detailed calculations =====
%$$
%\kappa_\alpha:=\dfrac {c_\alpha} {\Gamma(\alpha)}  \dfrac {\pi} {\cos \frac \pi 2\alpha} = \Gamma(2H+1)\sin (\pi H)
%$$
%$$
%\chi_{1,1}  =
%\frac 1 {2}  \Gamma\big(\frac {\alpha} {2}\big)\Gamma\big(2-\frac{\alpha}{2}\big)  =
%\frac 1 {2}  (1-\frac\alpha 2)\Gamma\big(\frac {\alpha} {2}\big)\Gamma\big(1-\frac{\alpha}{2}\big) 
%$$
%$$
%\chi_{0,1}  =
%\frac 1 {2} \Gamma\big(\frac {\alpha} {2}\big)\Gamma\big(1-\frac{\alpha}{2}\big)  
%$$
%$$
% \tfrac {3-\alpha} 2  \chi_{0,1}-   \chi_{1,1} = \frac 1 {4} \Gamma\big(\frac {\alpha} {2}\big)\Gamma\big(1-\frac{\alpha}{2}\big) 
%$$
%
and we obtain 
$$
\beta_1(H) =  {2^{2H-4}}\Gamma(2H+1).
$$

\medskip 
\noindent
{\em Case $H\in [\frac 2 3, \frac 3 4)$.} In this case $\big\lfloor \frac 1 {2H-1}\big\rfloor = 2$ and the sum in 
\eqref{Prob} contains two terms:
$$
\P(\|\widetilde B\|_2 \le \eps) \sim C(H) \eps\exp \Big(-\frac 1 8 \eps^{-2}-\beta_1 \eps^{4H-4}-\beta_2 \eps^{8H-6}\Big), \quad \eps\to 0,
$$
with coefficients 
\begin{equation}\label{gamma12}
\begin{aligned}
\beta_1 &=   \frac {1} {\sqrt 2   }  \eta_{0,1}
-      g_{1} \eta_{1,0} -\xi_1\\
\beta_2 &=   \frac {1} {\sqrt 2   }  \eta_{0,2}
-    g_{2} \eta_{2,0}
-      g_{1} \eta_{1,1}-\xi_2.
\end{aligned}
\end{equation}
To find $\xi_1$ and $\xi_2$, note that  
\begin{align*}
 y(r)^{-2} =\, &     \Big(y_0 + \sum_{j=1}^\infty y_j r^{j(2H-1)}\Big)^{-2}= \\
&
 y_0^{-2}-2y_0^{-3}  \sum_{j=1}^\infty y_j r^{j(2H-1)}+3 y_0^{-4} \Big(\sum_{j=1}^\infty y_j r^{j(2H-1)}\Big)^2+...   =\\
&
 y_0^{-2}
-2y_0^{-3}    y_1 r^{2H-1}
+\big(3 y_0^{-4} y_1^2   -2y_0^{-3}    y_2 \big)r^{2(2H-1)}
+O\big(r^{3(2H-1)}\big) 
\end{align*}
which yields
$$
\xi_1 = -2y_0^{-3}y_1  \quad \text{and}\quad \xi_2 = y_0^{-3}\big(3 y_0^{-1} y_1^2   -2     y_2\big).
$$
By \eqref{xieta} with $k=0$, we have 
$$
 y(r)^{-1} =    \Big(y_0 + \sum_{j=1}^\infty y_j r^{j(2H-1)}\Big)^{-1} = 
 y_0^{-1} -y_0^{-2}y_1 r^{2H-1} + O(r^{2(2H-1)})  
$$
and hence 
$$
\eta_{0,0} = y_0^{-1}\quad \text{and}\quad \eta_{0,1} = -y_0^{-2}y_1.
$$
Similarly, for $k=1$
\begin{align*}
 y(r)^{2H-2} =\, &  \Big(y_0 + \sum_{j=1}^\infty y_j r^{j(2H-1)}\Big)^{2H-2}= \\
&
 y_0^{2H-2}-(2-2H) y_0^{2H-3} y_1 r^{2H-1}+O\big(r^{2(2H-1)}\big)   
\end{align*}
which gives  
$$
\eta_{1,0}=y_0^{2H-2}\quad \text{and}\quad \eta_{1,1}=-(2H-2) y_0^{2H-3} y_1.
$$
Finally, \eqref{xieta} with $k=2$ yields $\eta_{2,0}=y_0^{4H-3}$. 
Plugging all these values into \eqref{gamma12} we get
\begin{align*}
\beta_1 &=  2y_0^{-3}y_1  -\frac {1} {\sqrt 2   }   y_0^{-2}y_1
-      g_{1} y_0^{2H-2}\\
\beta_2 &= -y_0^{-3}\big(3 y_0^{-1} y_1^2   +2     y_2\big) -\frac {1} {\sqrt 2   }  y_0^{4H-3}
-     g_{2} y_0^{4H-3} +     g_{1} (2-2H) y_0^{2H-3} y_1,
\end{align*} 
where $g_1$ and $g_2$ are found using \eqref{hkgk}. 
It is left to find $y_1$ and $y_2$. 

For $H\in [\frac 2 3, \frac 3 4)$ we have  
$\big\lfloor \frac 1 2\frac 1 {2H-1}\big\rfloor = 1$ and equation  \eqref{yreq} reads
$$
y/y_0 +    h_1 r^{ 2H-1}y^{ 2H }=1.
$$
Plugging expansion \eqref{yseries} we get
$$
\sum_{j=1}^\infty \frac{y_j}{y_0} r^{j(2H-1)}  +    h_1 r^{ 2H-1}\Big(y_0 + \sum_{j=1}^\infty y_j r^{j(2H-1)}\Big)^{ 2H }=0
$$
where $h_1$ is defined in \eqref{hkgk}. Comparing coefficients of powers $r^{2H-1}$ and $r^{4H-2}$ we obtain 
$$
y_1   = - h_1  y_0^{2H+1} \quad \text{and}\quad 
y_2  = -  2H h_1  y_0^{2H}y_1.
$$

\hfill $\blacksquare$
\end{example}

\subsection{The case $\boldsymbol{H\in}\mathbf{ (0, \frac 1 2)}$}
By Theorem \ref{eig-thm} the eigenvalues satisfy the same formula
$$
\lambda(n) =    \nu_n^{-2}      +    \kappa_\alpha   \nu_n^{ \alpha-3} 
$$
but this time, for $\alpha:=2-2H\in (1,2)$,  with $\nu_n = \pi n -\dfrac  \pi 2 q_\alpha + O(n^{1-\alpha})$ as $n\to\infty$, where 
$$
q_\alpha =1 - \frac{\alpha-1}{2} -\frac 2\pi  \arcsin\frac{\ell_{1-\alpha/2}}{\sqrt{1+\ell_{1-\alpha/2}^2}}.
$$
By the Taylor expansion 
\begin{align*}
\nu_n^{-2} & = \frac 1 {\pi^2} n^{-2} + \frac {q_\alpha}{\pi^2}n^{-3} + O(n^{-\alpha-2})\\
\nu_n^{\alpha-3} & = \frac 1 {\pi^{3-\alpha}} n^{\alpha-3} + \frac {3-\alpha}2\frac {q_\alpha}{\pi^{3-\alpha}}n^{\alpha-4} + O(n^{-3})
\end{align*}
and therefore  
$$
\lambda(n) =   \frac {\kappa_\alpha} {\pi^{3-\alpha}} n^{\alpha-3} + \frac 1 {\pi^2} n^{-2}               
+  \kappa_\alpha\frac {3-\alpha}2\frac {q_\alpha}{\pi^{3-\alpha}}n^{\alpha-4} + O(n^{-3}) := \phi(n)+ O(n^{-3}).
$$
As in the previous case, omitting the residual $O(n^{-3})$ term alters the exact asymptotics of small ball probabilities 
only by a multiplicative constant. The weight  function $\phi(t)$ has the form \eqref{phit} with  
\begin{equation}\label{csds_small}
\begin{aligned}
c_1 = \frac {\kappa_\alpha} {\pi^{3-\alpha}}, &\qquad  d_1 = 3-\alpha\\
c_2 =\frac {1} {\pi^2},&\qquad d_2 = 2\\
c_3 = \frac{\kappa_\alpha}{\pi^{3-\alpha}}\frac {3-\alpha}2 q_\alpha, & \qquad d_3 = 4-\alpha
\end{aligned}
\end{equation}
and the asymptotic expansions \eqref{I0v}, \eqref{I1lim} and \eqref{I2lim} read
\begin{align}
\label{I0lim_u}
I_0(u)   & =
\tfrac {1} 2 \big( 1  - \tfrac {1}2 q_\alpha    \big)\log u 
-\tfrac {3-\alpha} 2          \chi_{1,0}(2c_1)^{\frac 1 {3-\alpha}}u^{\frac 1 {3-\alpha}}
+         \sum_{k=1}^{m_\alpha} g_k
u^{ \frac{1-k(\alpha-1) }{3-\alpha}}
+ O(1)
  \\
\label{I1lim_u}  
I_1(u) & = 
- \tfrac 1 2     \chi_{1,0}  (2c_1 )^{\frac 1 {3-\alpha}}u^{\frac 1 {3-\alpha}}
-    \sum_{k=1}^{m_\alpha} h_k 
 u^{ \frac {1-k(\alpha-1)}{3-\alpha}}
+ O(1)    \\
\label{I2lim_u}
I_2(u)  &  = 
\tfrac 1 2\chi_{3,1}(2c_1u)^{\frac 1 {3-\alpha}}\big(1+o(1)\big) 
\end{align}
where sequences $h_k$ and $g_k$ are defined by the same formulas as in \eqref{hkgk}, evaluated at constants \eqref{csds_small}.

To find a suitable function $u(r)$ satisfying condition \eqref{cond}, consider equation (cf. \eqref{I1lim_u}) 
\begin{equation}\label{equ}
  \tfrac 1 2     \chi_{1,0}  (2c_1)^{\frac 1 {3-\alpha}} u^{\frac 1 {3-\alpha}}
+   \sum_{k=1}^{\left\lfloor \frac 1 2\frac 1 {\alpha-1}\right\rfloor} h_k 
 u^{ \frac {1-k(\alpha-1) }{3-\alpha}}=ur.
\end{equation}
As in the previous case, it has the unique solution $u(r)$ for any $r>0$ and it increases to $+\infty$ as $r\to 0$. 
By the choice of upper limit in the sum in \eqref{equ}, the power of $u(r)$ in the numerator of \eqref{cond} 
does not exceed $\frac 1 2\frac 1{3-\alpha}$ and hence \eqref{cond} holds in view of \eqref{I2lim_u}. 

Define new variable $y$
by the relation $u = (ry)^{-\frac{3-\alpha}{2-\alpha}}$, then it solves equation 
$$
y/y_0 +   \sum_{k=1}^{\left\lfloor \frac 1 2\frac 1 {\alpha-1}\right\rfloor} h_k  r^{ k\frac{ \alpha-1}{2-\alpha}   }y^{1 +k\frac{ \alpha-1}{2-\alpha} }=
1,
$$
where $1/y_0 =  \tfrac 1 2     \chi_{1,0}  (2c_1)^{\frac 1 {3-\alpha}}$. 
The function $r\mapsto y(r)$ is analytic in the vicinity of $r=0$ and can be expanded into powers of the small parameter
$r^{\frac{ \alpha-1}{2-\alpha}}$:
$$
y(r) = y_0 +\sum_{k=1}^\infty y_k r^{k\frac{ \alpha-1}{2-\alpha}}.
$$

Plugging \eqref{I0lim_u} and \eqref{I2lim_u} into \eqref{Pfla} yields 
\begin{align*}
& 
\P(\|\widetilde B\|^2_2\le r) \sim \\
&
    u(r)^{-\frac 1 4-\frac 1 2 \frac 1 {3-\alpha}+\frac {1} 2  ( 1  - \frac {1}2 q_\alpha)}  \exp\Big(
-\tfrac{3-\alpha}{y_0}u(r)^{\frac 1 {3-\alpha}}
+         \sum_{k=1}^{m_\alpha} g_k
u(r)^{ \frac{1-k(\alpha-1) }{3-\alpha}}
+ u(r) r\Big) \sim\\
&
    r^{\gamma_\alpha}  
\exp
\Big(
-\tfrac{3-\alpha}{y_0}(ry)^{-\frac{1}{2-\alpha} }
+         
\sum_{k=1}^{m_\alpha} g_k (ry)^{ -\frac{1-k(\alpha-1)2}{2-\alpha}}
+  (ry)^{-\frac{3-\alpha}{2-\alpha}}r 
\Big) \sim \\
&
   r^{\gamma_\alpha}  
\exp
\Big(
-\tfrac{3-\alpha}{y_0} \sum_{j=0}^\infty \eta_{0,j} r^{ \frac{j (\alpha-1)-1}{2-\alpha}}
+         
\sum_{k=1}^{m_\alpha} \sum_{j=0}^\infty g_k  \eta_{k,j}  r^{\frac{ (j+k)(\alpha-1)-1}{2-\alpha} }
+  \sum_{j=0}^\infty \xi_j r^{\frac{ j(\alpha-1)-1}{2-\alpha}} 
\Big)\sim \\
&
    r^{\gamma_\alpha}  
\exp
\Big(
-\beta_0 r^{-\frac{1}{2-\alpha}}  
 - \sum_{\ell=1}^{m_\alpha}\Big(  \tfrac{3-\alpha}{y_0}\eta_{0,\ell}  
-        
  \sum_{j=0}^{\ell-1} g_{\ell-j}  \eta_{\ell-j,j}      
-    \xi_\ell\Big)  
r^{\frac{ \ell(\alpha-1)-1}{2-\alpha}} 
\Big)
\end{align*}
where we defined 
$$
\gamma_\alpha := -\frac{3-\alpha}{2-\alpha}\Big(-\frac 1 4-\frac 1 2 \frac 1 {3-\alpha}+\frac {1} 2  \big( 1  - \frac {1}2 q_\alpha\big)\Big)
\quad \text{and}\quad
\beta_0 := \tfrac{3-\alpha}{y_0}\eta_{0,0}   -   \xi_0,
$$
and $\xi_j$ and $\eta_{k,j}$ are coefficients in the expansions
$$
y^{ -\frac{1-k(\alpha-1)}{2-\alpha}} =  \sum_{j=0}^\infty \eta_{k,j} r^{j\frac{ \alpha-1}{2-\alpha}}
\quad \text{and}\quad 
 y^{-\frac{3-\alpha}{2-\alpha}} = \sum_{j=0}^\infty \xi_j r^{j\frac{ \alpha-1}{2-\alpha}}.
$$
Replacing $\alpha$ with $2-2H$ and $r:=\eps^2$ and simplifying, we obtain the formula \eqref{Peps}:

% ======= simplification of $\beta_0$ ==========
%
%\cleaprage 
% We have $\xi_0 = y_0^{-\frac{3-\alpha}{2-\alpha}}$ and $\eta_{0,0}=y_0^{-\frac 1 {2-\alpha}}$ and hence 
%$$
%\beta_0 :=  (2-\alpha)y_0^{-\frac{3-\alpha}{2-\alpha}}    
%$$
%Since $c_1 = \frac {\kappa_\alpha} {\pi^{3-\alpha}}$ and
%$$
%\chi_{1,0} = \frac 1 {3-\alpha}\frac \pi {\sin \frac\pi{3-\alpha}}
%$$
%we have 
%$$
%y_0^{-1} =  \tfrac 1 2     \frac 1 {3-\alpha}\frac \pi {\sin \frac\pi{3-\alpha}}  (2\frac {\kappa_\alpha} {\pi^{3-\alpha}})^{\frac 1 {3-\alpha}} =
%   2^{-\frac {2-\alpha} {3-\alpha}}  \frac 1 {3-\alpha}\frac 1 {\sin \frac\pi{3-\alpha}}   \kappa_\alpha^{\frac 1 {3-\alpha}}
%$$
%and consequently 
%$$
%\beta_0 = (1-\tfrac \alpha 2)  
%\Big(  \frac 1 {3-\alpha}\frac 1 {\sin \frac\pi{3-\alpha}}   \Big)^{ \frac{3-\alpha}{2-\alpha}}
%\Big(  \kappa_\alpha^{\frac 1 {3-\alpha}}\Big)^{ \frac{3-\alpha}{2-\alpha}}=
%\frac{H}{(2H+1)^{\frac{2H+1}{2H}}}\left(\frac{\sin (\pi H)\Gamma(2H+1)}{\left(\sin \frac \pi {2H+1}\right)^{2H+1}}\right)^{\frac 1 {2H}}
%$$

% ====== simplification of $\gamma_\alpha$ ========
%\clearpage 
%\begin{align*}
%\gamma_\alpha = &
% \frac{1}{2-\alpha}\Big(\frac {3-\alpha} 4+\frac 1 2  -\frac {3-\alpha} 2  \big( 1  - \frac {1}2 q_\alpha\big)\Big) =\\
%&
% \frac{1}{2H}\Big(\frac {2H+1} 4+\frac 1 2  -\frac {2H+1} 2  \Big(    
%\frac {1}2 +  \frac{1-2H}{4}+  \frac 1\pi  \arcsin\frac{\ell_H}{\sqrt{1+\ell_H^2}}\Big)
%\Big)  =\\
%&
% \frac{1}{4H}\Big( \frac 3 4       
%+  H^2 -\frac {2H+1} {\pi}      \arcsin\frac{\ell_H}{\sqrt{1+\ell_H^2}} 
%\Big)
%\end{align*}
%which yields \eqref{gammaH} upon replacing $r:=\eps^2$. 

\begin{prop}\label{main-thm-smallH}
For $H\in (0,\frac 1 2)$ let $h_k$ and $g_k$ be the real sequences, defined by formulas \eqref{hkgk}, 
evaluated at the constants \eqref{csds_small}.  Then for any $r>0$ the equation 
$$
  y/y_0
+   \sum_{k=1}^{\left\lfloor \frac 1 2\frac 1 {1-2H}\right\rfloor} h_k  r^{ k\frac{ 1-2H}{2H}   }y^{k\frac{ 1-2H}{2H} +1}=
1
$$
with 
$$
y_0  = (2H+1) 
   \left(\frac{2^{ 2H }\left(\sin \frac \pi {2H+1}\right)^{2H+1}}{\sin (\pi H)\Gamma(2H+1)}\right)^{\frac 1 {2H+1}}
$$
has unique positive root $y(r)$, which can be expaned into series  
$$
y(r) = y_0 +\sum_{k=1}^\infty y_k r^{k\frac{1-2H}{2H}}
$$
convergent for all $r$ small enough.  Let $\xi_j$ and $\eta_{k,j}$ be coefficients of the power expansions  
$$
y(r)^{-\frac{2H+1}{2H}} =  \sum_{j=0}^\infty \xi_j r^{j\frac{  1-2H }{2H}}
\quad \text{and}\quad 
y(r)^{  \frac{k(1-2H)-1}{2H}} =  \sum_{j=0}^\infty \eta_{k,j} r^{j\frac{1-2H}{2H}}
$$
and define 
$$
\beta_\ell(H) :=    \frac{2H+1}{y_0}\eta_{0,\ell}  - \sum_{j=0}^{\ell-1} g_{\ell-j}  \eta_{\ell-j,j}  -  \xi_\ell.
$$
Then 
$$
\P\big(\|\widetilde B\|_2   \le \eps\big)   \sim      \eps^{\gamma(H)}\exp\bigg(
-\beta_0(H) \eps^{-\frac 1 {H}} - \sum_{\ell=1}^{\big\lfloor \frac 1 {1-2H}\big\rfloor}\beta_{\ell}(H)  \eps^{ \frac{ \ell(1-2H)-1}{ H}} 
 \bigg), \quad \eps\to 0
$$
where $\gamma(H)$ and $\beta_0(H)$ are given by  \eqref{gammaH} and \eqref{beta0H}.
\end{prop}

%\bibliographystyle{plain}
%\bibliography{/Users/Pavel/Dropbox/Pasha_Masha/bibliography/fBm}

\def\cprime{$'$} \def\cprime{$'$} \def\cydot{\leavevmode\raise.4ex\hbox{.}}
  \def\cprime{$'$} \def\cprime{$'$} \def\cprime{$'$}

\end{document}